\documentclass{amsart}
\usepackage{amssymb}
\usepackage{color}
\usepackage{float}
\usepackage[all,cmtip]{xy}
\usepackage{graphicx}
\usepackage{hyperref}
\usepackage{cancel}
\usepackage{pstricks}
\usepackage{mathabx}
\usepackage{tikz-cd}

\makeatletter
\newsavebox{\@brx}
\newcommand{\llangle}[1][]{\savebox{\@brx}{\(\m@th{#1\langle}\)}%
  \mathopen{\copy\@brx\kern-0.5\wd\@brx\usebox{\@brx}}}
\newcommand{\rrangle}[1][]{\savebox{\@brx}{\(\m@th{#1\rangle}\)}%
  \mathclose{\copy\@brx\kern-0.5\wd\@brx\usebox{\@brx}}}
\makeatother

\newcommand{\Acal}{{\mathcal{A}}}
\newcommand{\Bcal}{{\mathcal{B}}}

\newcommand{\Fcal}{{\mathcal{F}}}

\newcommand{\Hcal}{{\mathcal{H}}}
\newcommand{\Ical}{{\mathcal{I}}}

\newcommand{\Lcal}{{\mathcal{L}}}

\newcommand{\Ncal}{{\mathcal{N}}}

\newcommand{\Rcal}{{\mathcal{R}}}

\newcommand{\Ucal}{{\mathcal{U}}}
\newcommand{\Vcal}{{\mathcal{V}}}
\newcommand{\Wcal}{{\mathcal{W}}}

\newcommand{\DD}{\mathbb{D}}

\newcommand{\NN}{\mathbb{N}}

\newcommand{\RR}{\mathbb{R}}

\newcommand{\ZZ}{\mathbb{Z}}

\newcommand{\Cont}{\operatorname{Cont}}

\newcommand{\Diff}{\operatorname{Diff}}

\newcommand{\Ham}{\operatorname{Ham}}

\newcommand{\codim}{\operatorname{codim}}

\newcommand{\sub}{\operatorname{Sub}}
\newcommand{\Sub}{\operatorname{Sub}}
\newcommand{\Op}{\operatorname{Op}}

\newtheorem{theorem}{Theorem}
\newtheorem*{theorem*}{Theorem}
\newtheorem{proposition}{Proposition}[section]
\newtheorem{lemma}[proposition]{Lemma}

\newtheorem{corollary}[theorem]{Corollary}

\theoremstyle{definition}
\newtheorem{definition}[proposition]{Definition}
\newtheorem{example}[proposition]{Example}

\theoremstyle{remark}
\newtheorem{remark}[proposition]{Remark}
\newtheorem{question}[proposition]{Question}

\numberwithin{equation}{section}

\newcommand{\bparagraph}[1]{\vspace{1em}\noindent\textbf{#1}.\hspace{0.2cm}}

\begin{document}

\title[]{On the intersections of projected Hamiltonian orbits in cotangent bundles}

\author{Lucas Dahinden, Jacobus de Pooter}


\date{\today}

\begin{abstract}
We study the generic behavior of Hamiltonian trajectories on a regular level set in the cotangent bundle, after projection to the base. We prove that for a generic submersive level set, projected trajectories have discrete (self-)intersections. Additionally, fixing end-point fibers, we prove that all intersections can be perturbed away if the base has dimension $\geq 3$. In particular, this applies to periodic orbits, and both results hold for Reeb flows on fiber-wise star-shaped hypersurfaces, including non-reversible Finsler flows, which answers a question of Rademacher. In the proof we make use of a multi-jet transversality theorem.
\end{abstract}

\maketitle

\tableofcontents

\section{Introduction}\label{sec:intro}

\subsection{Context and Background}
On a closed Riemannian manifold, there always exist closed geodesics. In the non-simply connected case this follows by length minimization in a non-trivial homotopy class, and in the simply connected case this is a classic result by Ljusternik and Fet~\cite{LF}. Given two points, there exists also a geodesic connecting them (a chord), which is proved by similar techniques. A geodesic is called simple if it admits an injective parametrization. The length-minimal non-contractible geodesic in the non-simply connected case is always simple, as a crossing may be used to split the curve into two shorter ones, at least one of which is non-contractible. In the simply connected case, the existence of a simple closed geodesic is an open problem~\cite{AHS}. It is a classic consequence of Morse theory that the loop space on a manifold (the space of closed paths or the space of paths from a point to the other) is homotopy equivalent to a CW-complex, which can be constructed from closed geodesics and chords, cf.~\cite{M}, which guarantees a wealth of closed geodesics chords. As is pointed out in~\cite[Problem C]{T}, these geodesics may be not geometrically distinct. The hunt for geometrically distinct geodesics is thus a challenge that goes beyond the simple homological construction, see e.g.~\cite{Franks, CM24, AM} for compact and non-compact settings in which arbitrarily many geometrically distinct geodesics exist. In contrast, a complementary example can be found in non-reversible Finsler geometry, where the Bangert sphere only has two geometrically distinct geodesics~\cite{B}.

Recent theorems by Rademacher~\cite{R23,R24} treat the problem for generic metrics: He proves that for generic metrics periodic geodesics and chords do not intersect (each other and themselves). We loosely and collectively state the theorems as follows.
\begin{theorem*}[Rademacher]
    Let $M$ be a closed manifold of dimension $\geq 3$ and let $p,q\in M$ be two points, not necessarily distinct. For a $C^r$-generic set of Riemannian metrics ($r\geq 2$), geodesic segments connecting $p$ to $q$ have no self-intersection, and geometrically distinct segments have no interior intersections. 

    Moreover, for a $C^r$-generic set of reversible Finsler metrics ($r\geq 4$), the same holds true.
\end{theorem*}
As a corollary, Rademacher shows \cite[Corollary 1]{R24} that generically there are infinitely many geometrically distinct geodesic chords connecting the given end-points.

Rademachers proof of the above theorem relies at some point on a simple observation, which we state here loosely as follows. 
\begin{center}
    If two regular curves intersect at accumulating times, \\then they are tangent. 
\end{center}
In Lemma~\ref{lem:transverseQ}, we show that we can even conclude tangency to arbitrary order. If these two curves are Riemannian geodesics, then the existence of an injectivity radius lets us conclude that they coincide. 

Rademachers proof has essentially three steps:
\begin{enumerate}
    \item The geodesics of interest are non-degenerate for a generic metric. In particular, for a length $L$ the set of geodesics of interest of length $\leq L$ is finite. 
    \item By the simple observation, tangent geodesics coincide. Thus, all intersections are transverse, thus isolated and there are only finitely many.
    \item Perturb the metric to resolve the finitely many intersections of the finitely many geodesics. Do repeat for every integer $L$ to get a genericity statement. 
\end{enumerate}
In Rademachers proof, the difficult step is the perturbation (3), which involves an implicit function theorem. 

The theorems still work well if the situation is generalized to reversible Finsler dynamics. However, for non-reversible Finsler dynamics there are tangent geodesics that do not coincide. This led Rademacher to ask:
\begin{question}[Rademacher]
    Does this theorem generalize to non-reversible Finsler metrics?
\end{question}

\subsection{Generalization}
We aim here to generalize Rademachers theorem beyond the class of Finsler metrics to a broader range of Hamiltonian dynamical systems. Note that non-intersecting geodesics are in particular geometrically distinct. For a Hamiltonian $H:T^*Q\to \RR$, we focus on the energy level set $\Sigma=\Sigma_H= H^{-1}(0)$. 
\begin{definition} 
    We denote by $\Ham_s$ the set of Hamiltonians $H\in C^\infty(T^*Q,\RR)$ such that the following properties hold.
    \begin{itemize}
        \item The zero level set is regular, i.e., $dH|_{\Sigma_H} \neq 0$.
        \item  The projection 
    $$\pi_Q:\Sigma=H^{-1}(0)\to Q$$
    is a submersion.
    \end{itemize}
\end{definition}
These conditions have important consequences for $\Sigma$. Regularity of the level set implies that $\Sigma$ is a manifold. The submersivity implies that the projection to the base manifold of the Hamiltonian vector field does not vanish $d\pi_Q X_H|_\Sigma \neq 0$ (see Lemma \ref{lem:submersiveHam}).
\begin{remark}
    It would also be convenient to demand that $Q$ is compact and $H$ is proper, which is in many applications the case. Then, $\Sigma$ would be a compact set, and together with submersivity we could conclude that it is a fiber bundle with pairwise diffeomorphic fibers. The main benefit of compactness of $\Sigma$ would be automatic completeness of the Hamiltonian flow. Our main result is local in nature, so that the comfort of completeness can be approximated by an exhaustion by compact sets. 
\end{remark}

In the following statements, we equip $C^\infty(T^*Q,\RR)$ with either the weak or strong Whitney $C^k$-topologies. The results hold for either. See for example \cite[Chapter 2]{GG73} for their definitions and properties. In the results below, $C^k$ residual means that the set of Hamiltonians for which the statement is true, is a countable intersection of $C^k$-open and $C^\infty$ dense set. Since the set of Hamiltonians is a Baire space, this intersection is still $C^\infty$ dense. We prove the following generalization of Rademachers theorem. 


\begin{theorem}\label{thm:Main1}
    Let $Q$ be a manifold of dimension $\geq 3$ and fix two points $q_1, q_2$. There is a set of Hamiltonians $H\in\Ham_s$ that is $C^2$-residual in the weak and strong Whitney topology such that the following statements are true:
    \begin{itemize}
        \item For any two distinct flow lines $\gamma_i:I_i\to\Sigma$ of the Hamiltonian vector field $X_H$ that connect the fibers $\Sigma_{q_1}$ and $\Sigma_{q_2}$ the projections to the base manifold $\pi_Q\circ \gamma_i$ intersect exclusively in the end points $\pi_Q\circ\gamma_1\cap \pi_Q\circ\gamma_2=\{q_1,q_2\}$. 
        \item No flow line $\gamma:I\to\Sigma$ of the Hamiltonian vector field $X_H$ that connects the fibers $\Sigma_{q_1}$ and $\Sigma_{q_2}$ admits a self-intersection: $\pi_Q\circ \gamma$ is injective (except at the end points if they coincide).
        \item For any pair of geometrically distinct periodic flow lines $\gamma_i:\RR/T_i\ZZ\to\Sigma$ of the Hamiltonian vector field $X_H$ the projections to the base manifold $\pi_Q\circ\gamma_i$ do not intersect.
        \item For any periodic flow line $\gamma:\RR/T\ZZ\to\Sigma$ of the Hamiltonian vector field $X_H$ that is not a multiple cover of the projection to the base manifold $\pi_Q\circ\gamma_i$ does not self-intersect: $\pi_Q\circ \gamma$ is injective.
    \end{itemize} 
\end{theorem}

The main difficulty in proving this theorem is step (2): it is easy to construct examples of flow lines with accumulating intersection points. Thus, we aim for something weaker than an injectivity radius, and we expect it to hold only generically.

\begin{definition}
    We say that $H\in\Ham_s$ has \emph{discrete intersections} at the 0-energy level $\Sigma_H$ if for any two Hamiltonian flow lines $\gamma_1:I_1\to\Sigma_H,\gamma_2:I_2\to \Sigma_H$ that do not coincide up to time shift the set of intersection times $$\{(t_1,t_2)\in I_1\times I_2 \mid \pi_Q(\gamma_1(t_1))=\pi_Q(\gamma_2(t_2))\}$$
    is discrete.
\end{definition}

Note that discreteness of intersections does not propose a lower bound on the distance between intersections, and the time of first intersection is also not continuous. This is illustrated by the following example.
\begin{example}
    The $s-$family of lines $\gamma_s(t)=(t, st-t^2)$ in $\RR^2$ intersects the $x$-axis for each $s>0$ twice: $\gamma_s\cap\{y=0\}=\{(0,0),(s,0)\}$. The distance between the intersections is not bounded away from zero. It is easy to construct a Hamiltonian in $T^*\RR^2$ such that the lines are projected flow lines.
\end{example}

The essential step in the proof of Theorem~\ref{thm:Main1} is to show that generic Hamiltonians have discrete intersections.

\begin{theorem}\label{thm:discreteIntersections}
    Let $Q$ be a manifold of dimension at least 3. Then, a  subset of $\Ham_s$ that is $C^7$-residual in the weak and strong topology Whitney topology has discrete intersections at $\Sigma$. 
    If the dimension of $Q$ is two, then a  subset of $\Ham_s$ that is $C^6$-residual in the weak and strong topology Whitney topology has discrete intersections at $\Sigma$. 
\end{theorem}

We prove Theorem~\ref{thm:Main1} below assuming Theorem~\ref{thm:discreteIntersections} and sketch the proof of Theorem~\ref{thm:discreteIntersections}.

\subsection{Notable special cases}
We note that the subset $\Cont\subseteq \Ham_s$ of Hamiltonians whose hypersurface $\Sigma$ is fiberwise starshaped with respect to 0 is $C^k$-open ($k\geq 1$). These flows are (up to reparametrization) precisely the contact flows on the spherization $S^*Q$.

A further $C^k$-open ($k\geq 2$) subset $\Fcal\subseteq\Cont$ consists of Hamiltonians such that $\Sigma$ is fiberwise analytically convex (with 0 in its interior). These flows are (up to reparametrization) precisely the Finsler geodesic flows: the Finsler norm can be recovered as $F(q,v)=\max\{p(v)\mid p\in \Sigma_q\}$.

If a set $A\subseteq \Ham_s$ is $C^k$-residual and $B\subseteq \Ham_s$ is $C^k$-open, then $A\cap B$ is $C^k$-residual in $B$. Thus, we immediately gain the following corollary.

\begin{corollary}\label{cor:main}
    The conclusions of Theorem~\ref{thm:Main1} and Theorem~\ref{thm:discreteIntersections} also hold if one replaces $\Ham_s$ by $\Cont$ or $\Fcal$. 
\end{corollary}

\subsection{Many geometrically distinct geodesics}
In many cases, one can construct filtered homologies from chords. As pointed out before, in the geodesic case, one can construct a Morse theory from the energy functional, which coincides with the loop space homology filtered by energy~\cite{M}. The growth of this homology induces growth of number of chords since they are the generators of the homology. By a result of Gromov~\cite{Gromov}, see also \cite[Theorem 5.10]{paternain1999geodesic}, there is a lower bound to this growth coming from the growth of the homology filtered by degree, which is at least linear on closed manifolds.

In the case of generic compact, exactly fillable hypersurfaces, one can define symplectic homology relative Legendrians that are fillable by exact Lagrangians, or the corresponding positive Rabinowitz--Floer homology~\cite{CF09} (which is the same homology,~\cite{CFO}). They are generated by chords between the selected Legendrians, which are provided by the fibers of $\Sigma$ over the selected points. If the hypersurface is fiberwise starshaped, then this homology coincides with loopspace homology ~\cite{AS}. If the loop space homology grows exponentially in degree, then automatically the corresponding growth of chords in length is positive by a sandwiching argument, for more details see~\cite{MacSch}.  

\begin{corollary}
    For a closed manifold $Q$ and two fibers of the cotangent bundle, for a $C^2$ residual set of exactly fillable compact hypersurfaces, the number of geometrically distinct chords connecting them grows in length at least as much as the dimension of symplectic homology filtered by length.

    If furthermore the hypersurface is fiberwise starshaped, then the growth is at least linear. If the loop space homology grows exponentially in degree, then also the number of geometrically distinct chords grows exponentially in length. 
\end{corollary}

\subsection{Proofs}

\bparagraph{Proof of Theorem~\ref{thm:Main1}} We mimic the proof of Rademacher. Note that since we only consider closed orbits and chords between fibers, the possible non-completeness of the flow is of no concern. We choose a time $T\in\NN$. 

\paragraph{(Step 1)} In the Hamiltonian setting it is well known that a $C^2$-residual set $\Hcal_{reg}$ of Hamiltonians has only non-degenerate periodic orbits and only non-degenerate chords from the fiber $\Sigma_{q_1}$ to $\Sigma_{q_2}$. Therefore, there are only finitely many orbits to consider. We will make use of the following stronger statement: For each relatively compact open set $K\subseteq T^*Q$ the set $\Hcal_{reg,K}$ of Hamiltonians with only non-degenerate periodic orbits and chords in $K$ of length at most $T$ is $C^2$ open and dense. Furthermore, as long as $H\in\Hcal_{reg,K}$, these orbits depend continuously in $C^0$ topology of the orbits on $H$ in $C^2$ topology (up to growing longer or shorter than $T$ and wandering out of or into $K$). Note that here it is important that we stay in $\Hcal_{reg}$: if we pass through non-regular Hamiltonians, the critical points may jump.

\paragraph{(Step 2)} Due to Theorem~\ref{thm:discreteIntersections}, for a $C^7$-residual subset of $\Hcal_{reg}$ the orbits of interest intersect themselves and each other only finitely many times. In particular, this set is $C^\infty$-dense.
\paragraph{(Step 3)} We need to show that we can smoothly and locally perturb $H$ to resolve the intersection, which provides a $C^\infty$-dense set $\Hcal_0$ of Hamiltonians such that all orbits of interest have no intersection. 

Consider an intersection point $q\in Q$ of finitely many pieces of orbit $\gamma_i=\gamma|_{I_i}$ that we intend to resolve. We shift the parametrization such that $\pi_Q\gamma_i(0)=q$. This intersection point is isolated, so we can perturb in a small neighbourhood where $q$ is the only intersection point.

We try to displace $\gamma_1$ from the other segments. For this, we need to choose where $\gamma_1$ should go instead, so we search alternative curves $\delta_i$ that do not intersect the other curves and $\delta_i\to\gamma_1$. To this end, trivialize a neighbourhood of $\pi_Q\gamma_1$ by its normal bundle $\nu\to\pi_Q\gamma_1$ and choose polar coordinates $(r,\theta)$ in the fibers such that $(t,r,\theta)$ parametrizes the tubular neighbourhood. Fixing the normal coordinates produces a family of curves $\gamma_{r,\theta}(t)=(t,r,\theta)$ that are `parallel' to $\pi_Q\gamma_1$. We interpret this as a family (parametrized by $\theta$) of 1-parameter families (parametrized by $r>0$). 

There is a $\theta$, such that the total image of the 1-parameter family $\bigcup_{r>0,t}\gamma_{r,\theta}(t)$ intersects the other orbits $\pi_Q\gamma_i$ in a set of times with empty interior (as there are uncountably many choices of $\theta$ and every open set contains a rational time). For this $\theta$, there is a sequence $r_k\to 0$ such that $\gamma_{\theta,r_k}$ intersects none of the the other curves $\pi_q\gamma_i$. 

To reconnect these curves with $\gamma_1$, we fix $\varepsilon,\epsilon>0$ small enough such that the $\epsilon$-balls around $\gamma_1(\pm\varepsilon)$ do not intersect any other curve and we choose a smooth 1-parameter family of curves $\delta_r$ such that the following properties hold
\begin{itemize}
    \item $\delta_r(t)=\pi_Q\gamma_1(t)$ for $t\notin (-\varepsilon-\epsilon,\varepsilon+\epsilon)$.
    \item $\delta_r(t)=\gamma_{\alpha,r}(t)$ for $t\in (-\varepsilon+\epsilon,\varepsilon-\epsilon)$.
\end{itemize}

Choose a smooth family of diffeomorphisms $\Phi:Q\times[0,1)\to Q$ such that $\Phi_r\circ \pi_Q \gamma_1=\delta_r$. And such that $\Phi-id$ is supported in an $\epsilon$-neighbourhood of $\gamma_i(-2\varepsilon,2\varepsilon)$. Then, the pullbacks $\Phi^*_r(\lambda)=\lambda\circ D\Phi_r$ is a 1-parameter family of symplectic diffeomorphisms of $T^*Q$. Pulling back the Hamiltonian by the inverse of this diffeomorphism
$$\tilde H_r = ((\Phi^*_r)^{-1})^*H$$
and the flow lines of $\tilde H_r$ are images under $(\Phi^*_r)^{-1}$ of flow lines of $H$. Finally, we choose nested neighbourhoods $\gamma_1|_{(-3\varepsilon,3\varepsilon)}\subseteq\Ncal_1\subseteq\Ncal_2\subseteq T^*Q$ of $\gamma_1$ that do not intersect any of the other curves $\gamma_i$ and choose a cutoff function $\chi:T^*Q\to \RR$ such that $\chi|_{\Ncal_1}=1,\chi|_{\Ncal_2^c}=0$. The desired family of Hamiltonians is 
$$H_r= \chi H_r +{1-\chi}H.$$
For small $r, H_r$ is $C^\infty$-close to $H$ and for $r=r_k$ it displaces $\pi_Q\gamma_1$ through $\delta_r$ from the intersection as required. Repeat the construction until no intersections are left. 

\paragraph{(Conclusion)} We note that if $H\in \Hcal_0\cap\Hcal_{reg,K}$, then in a $C^2$-neighbourhood of $H$, the orbits come in families that depend continuously on $H$. The property of non-intersection of curves is $C^0$-open in the space of flows, thus one can choose these neighbourhoods so small that they lie completely in $\Hcal_0\cap\Hcal_{reg,K}$. In other words, $\Hcal_0\cap \Hcal_{reg,K}$ is $C^2$ open and $C^\infty$ dense. By taking intersections by $T\in\NN$ and an exhaustion of $T^*Q$ by compact sets $K$, we obtain the result. 
\qed

\begin{remark} 
We note that Theorem~\ref{thm:Main1} is wrong for lower regularities by referring to the Riemannian case. For $C^0$ regularity on the metric $g$, the result is easily seen to be wrong as we can produce closed orbits and chords wherever we want by $C^0$ perturbation by methods explained in~\cite[Example 43]{ALMM}. 
To see that the statement is wrong for $C^1$ regularity we inspect the case of chords: If we fix a starting fiber and look at all orbits of length $\leq T$ (but that end anywhere), there may be many mutually- and self-intersecting ones whose existence cannot be perturbed away. However, we are only interested in the selected few that land on the target (non-conjugate) fiber. Non-intersection (away from the diagonal) of these selected chords is an open property in $C^0$-topology of the selected chords. But the selected chords are not continuous in $C^1$-small changes of $g$: we can manipulate the curvature along the chord arbitrarily by adding $C^1$-small but $C^2$-large bumps. This way we construct in any $C^1$ neighbourhood of $g$ a $g'$ with the same chord, but for which the geodesic spray is focused at the target fiber, making it a conjugate point. Further $C^1$ perturbation then makes sure the chord bifurcates into multiple ones (or even an infinite family), the mutual intersection pattern of which we have no control over. Thus, the mutual intersection statement is wrong for $C^1$.
\end{remark}

\subsection{Sketch of proof of Theorem~\ref{thm:discreteIntersections}}
We first reformulate the problem for a fixed Hamiltonian $H\in \Ham_s$. For each point in $\Sigma$ we get a flow line $\varphi^t_H(p,q)$ line going through this point. Restricting the flow line to a small time interval, it is a submanifold with base point $\varphi^0_H(p,q)$. The 1-dimensional submanifolds with base point form a bundle $\sub(\Sigma,1)$ over $\Sigma$ and $H$ thus provides a section
$$\sigma_H: \Sigma \longrightarrow \sub(\Sigma,1).$$
Actually, there is a subset  $\sub_{\pi_Q}(\Sigma,1)$ of sub-manifolds for which the base projection $\pi_Q$ is an immesion. Our section $\sigma_H$ has image in  $\sub_{\pi_Q}(\Sigma,1)$ since it avoids curves that are tangent to the fiber due to the assumption that the projection $\pi_Q\colon\Sigma\to Q$ is a submersion. Additionally, $\pi_Q$ induces a map 
$$\sub_{\pi_Q}(\Sigma,1)\stackrel{\pi_Q}\longrightarrow \sub(Q,1)$$
to the 1-dimensional submanifold bundle of $Q$. Further, we consider the bundle of jets of submanifolds $J^k_{\pi_Q}(\Sigma_1)$ and $J^k(Q,1)$ with corresponding maps that fit into the following diagram.
\[
  \begin{tikzcd}
    \Sigma \arrow{r}{\sigma_H} \arrow{dr}{\sigma_H^k}&  \sub_{\pi_Q}(\Sigma,1) \arrow{r}{\pi_Q} \arrow{d}{j^k} & \sub(Q,1) \arrow{d}{j^k} \\
         &  J^k_{\pi_Q}(\Sigma,1) \arrow{r}{\pi_Q^k} & J^k(Q,1)
  \end{tikzcd}
\]
We will establish a Lemma (see \ref{lem:transverseQ} for a more precise formulation) that if two one dimensional submanifolds that intersect have a $k$-jet that does not agree, then the intersection is isolated. Thus, the discreteness of intersections is proved for $H$ if there is a $k$ such that $j^k\circ \pi_Q\circ \sigma_H=\pi_Q^k\circ\sigma_H^k$ is injective.

This statement is equivalent to saying that the multisection $$\pi_Q^k\circ\sigma_H^k\times \pi_Q^k\circ\sigma_H^k \colon\Sigma\times\Sigma\to J^k(Q,1)\times J^k(Q,1)$$
only hits the diagonal of the target $\Delta_{J^k(Q,1)}$ from the diagonal of the base $\Delta_\Sigma$, or in other words 
$$(\pi_Q^k\circ\sigma_H^k\times \pi_Q^k\circ\sigma_H^k)^{-1}(\Delta_{J^k(Q,1)})=\Delta_\Sigma.$$
We exploit that this map factors through $J^k_{\pi_Q}(\Sigma,1)\times J^k_{\pi_Q}(\Sigma,1)$, so we reformulate the property as
$$\sigma_H^k\times \sigma_H^k (\Sigma\times \Sigma \setminus\Delta_\Sigma) \cap (\pi_Q^k\times\pi_Q^k)^{-1}(\Delta_{J^k(Q,1)}) =\emptyset.$$
The space $(\pi_Q^k\times\pi_Q^k)^{-1}(\Delta_{J^k(Q,1)})$ is a submanifold since $\pi_Q^k$ is submersive. 
We prove in Theorem~\ref{thm:homopodaltransversality} that for a $C^{k+2}-$residual set of Hamiltonians $H\in\Ham_s(Q)$ the following is true:
\begin{itemize}
    \item The intersection is (away from the diagonal) transverse 
        $$j^k\sigma_H\times j^k\sigma_H (\Sigma\times \Sigma\setminus\Delta_\Sigma) \pitchfork (\pi_Q\times\pi_Q)^{-1}(\Delta_{J^k(Q,1)})$$
        and therefore a submanifold.
    \item The dimension of the submanifold is $(3-k)(n-1)+1.$ 
\end{itemize}
Thus, we conclude that for $k$ large enough the dimension is negative and the intersection empty. More precisely, for $n=2$ this is $k\geq 5$, for $n>2$ this is $k\geq 4$, which is established in Corollary~\ref{cor:dimensioncount}.

Let us sketch how to arrange transversality by perturbing $H$, more details are provided in the main text. The first challenge is that $\Sigma_H$ depends on $H$, so the perturbed flow lives on a different space. Luckily, most of $H$ is of no concern to us: $X_H|_\Sigma$ is a section of the characteristic line bundle $\chi_\Sigma = \ker\omega|_\Sigma$. The bundle $\chi_\Sigma$ only depends on the shape of $\Sigma$, and since our considerations are oblivious to reparametrizations of the flow lines, we may as well talk about $\Sigma$ as the main object of interest. Choosing a normal bundle $\nu$ over $\Sigma$, nearby hypersurfaces are described as a section $f:\Sigma\to \nu$, i.e., a function. By choosing coordinates $q$ on $Q$ and extending them to dual coordinates $p,q$ on $T^*Q$, we can describe a convenient choice of normal bundle locally in $V\subseteq\Sigma$ by choosing a constant (w.r.t.\ the coordinates) section $p_0$ of $T^*Q$ such that $V$ is transverse to the radial lines in $T^*Q$ emenating from $p_0$. In this local description, $V$ equipped with $\lambda-p_0$ is a contact manifold and the flow on perturbations of $V$ is the contact Hamiltonian flow with respect to the dilation function. This allows us to use the machinery of contact geometry to show that the map from jets of Hamiltonians to jets of submanifolds is submersive. Finally, the residual transversality is concluded by the multijet transversality theorem from~\cite{G25}. 
\qed

\bparagraph{Outline of the paper} 

\subsection*{Acknowledgements} The first author was funded by the NWO grant Vidi.223.118 “Flexibility and rigidity of tangent distributions”. The second author was funded by the Deutsche Forschungsgemeinschaft (DFG, German Research Foundation) – Project-ID 281071066 – TRR 191. We want to thank Hans-Bert Rademacher for making us aware of the project, and Alberto Abbondandolo and Alvaro del Pino insightful discussions and suggestions.


\section{Preliminaries}


\subsection{Jet bundles}
Let us first revisit jet bundles, as introduced for example in \cite[Chapter 1]{EM02}, \cite[Chapter 1.1]{Gr10} and \cite[Chapter II]{GG73}. Let $E\to M$ be a smooth fiber bundle, with $\dim M=m$. We think of $k$-jets $J^k(E)$ as follows: Consider the space of smooth local sections with a choice of marked point $$C^\infty(E)_{loc}\times M=\{(\gamma,x)\mid \gamma:\Op(x)\to E\}.$$ Note that we focus on local sections in an arbitrarily small neighbourhood $\Op(x)$ of the marked point, as we are not interested in how (or if at all) the section extends beyond that neighbourhood. For each point $x\in M$ there is an equivalence relation on $C^\infty(E)\times\{x\}$ that identifies sections that are tangent to $k$-th order at $x$. The jet bundle is then $J^k(E)=C^\infty(E)\times M/\sim$. We topologize $J^k(E)$ by taking the quotient of the topology induced by the $C^k$-norm (which gives the same topology as the weak Whitney $C^k$-topology) on $C^\infty(E)$. The jet bundle can be seen as a bundle over the base space $J^k(E)\to M;[\gamma,x]\mapsto x$ or as a bundle over the total space $E$, $J^k(E)\to E;[\gamma,x]\mapsto (x,\gamma(x))$. 

We make this more precise in local coordinates: tangency to order $k$ means that all order $r$ derivatives agree at $x$, where $0\leq r\leq k$ runs through all options. Given a small enough neighbourhood $U$, we fix a trivialization of $E|_U$ as a product $U\times E_x$ and choose a neighbourhood $V\subseteq E_x$ in the fiber so that we parametrize $U\times V$ with a product map $\phi\times\psi:U\times V\to U'\times V'\subseteq \RR^m\times\RR^{\dim E_x}$ that is a diffeomorphism onto its image. A local section $\sigma:U\subseteq M\to U\times V$ is represented by a function $f:U'\to V'$ that is defined by $f(x)=\psi\circ\sigma\circ\phi^{-1}(x)$. This gives rise to local coordinates of the jet bundle $\Psi^k:J^k(E)|_{U\times V}\to \RR^N$ by declaring the jet of $f$ at $x$ to be 
$$\Psi\circ j^k\sigma(\phi^{-1}(x))=j^kf(x)=\left(x,f(x),f^{(1)}(x),\ldots,f^{(k)}(x)\right)$$
where $f^{(r)}=(\partial^{\alpha_1}_{x_1}\cdots\partial^{\alpha_m}_{x_m} f)$ is the collection of partial derivatives with the multi-index $\alpha=(\alpha_1,\ldots,\alpha_n)$ running through all non-negative integers with $\sum_{i=1}^m\alpha_i=r$. Note that we are taking ordered derivatives so that there is no redundancy.

We note that the jet of $f$ at $x$ depends only on a small neighbourhood of $x$: if $f|_{U''}=g|_{U''}$ for some open set $x\in U''\subset U'$, then $j^kf(x)=j^kg(x)$, so restricting the coordinates to a smaller set does not change the coordinates for the jet bundle. This is why we will use the notation $\Op(x)$ for an arbitrary small neighbourhood of $x$, that we possibly make smaller if needed. 

Note that the distinction of strong and weak Whitney topologies becomes relevant when topologizing the space of sections of the jet bundles, but not when topologizing the jet bundles.


\subsection{Jets of submanifolds}
We will make use of \emph{the jet bundle of submanifolds}, denoted by $J^k(M,n)$. The main reference we use is \cite{Ma03}, which also contains a historical overview of the introduction and application of the concept. The definition is originally due to Ehresmann \cite{Eh52}. 

Geometrically, the concept is very similar to jet bundles: We consider the space $\Sub(M,n)$  of $n$-dimensional submanifolds $N$ with a marked point $*\in N$ and intend to take a quotient modulo order $k$ tangency. We consider only neighbourhoods of the marked point. As we can restrict a chart, we can think of a marked $n$-dimensional manifold locally as an open neighbourhood $\Op(0)\subset \RR^n$ of the marked point $0\in \RR^n$. When needed, we allow ourselves to restrict $\Op(0)$ to an even smaller neighbourhood of $0$, which does not change the tangency behaviour at 0.  

\subsubsection{Bundle of submanifolds}
We denote the space of immersions of $\Op(0)$ into $M$ by
$$C^+(\Op(0),M)=\{\gamma\in C^\infty (\Op(0), M), d\gamma \mbox{ injective}\}.$$
Every immersion can be made into an embedding by shrinking the neighbourhood $\Op(0)$. The space $C^+(\Op(0),M)$ is a bundle over $M$ by the base point map $\gamma\mapsto\gamma(0)$. Reparametrizations $\Diff_0(\Op(0))$ of $\Op(0)$ that fix $0$ act on elements of $C^+$. 
\begin{definition}
    The bundle of local marked $n$-dimensional submanifolds is the quotient space
    $$\sub(M,n)=C^+(\Op(0),M)/\Diff_0(\Op(0)).$$
    The bundle map $\sub(M,n)\to M$ is the evaluation at zero of some - and thus any - representative $[\gamma]\mapsto \gamma(0)$. 
\end{definition}
We endow this space with the $C^k$ topologies as follows: $[\gamma_i]$ converges to $[\gamma]$ in $C^k$ if for a parametrization $\gamma:\Op(0)\to M$ there are parametrizations $\gamma_i:\Op(0)\to M$ (all defined on a common domain) that converge on a compact neighbourhood in the $C^k$-norm to $\gamma$. 

We find local coordinates for the bundle of submanifolds by choosing local coordinates of $M$ that are especially adapted to the submanifold. Let $\gamma:\Op(0)\to M$ be a local parametrization of a marked submanifold $N$ as before and let $V$ be a neighbourhood of $\gamma(0)$ containing the image of $\gamma$. 
\begin{definition}
    Coordinates $(x_j,\nu_l):V\subseteq M\to\RR^m$ of $M$, $1\leq j\leq n$ and $1\leq l\leq \dim m-n$, are \emph{divided} with respect to $N$ if $N$ is a graph manifold over the local coordinates $x_j$: 
    $$x\circ \gamma =(x_1\circ\gamma,\ldots,x_n\circ\gamma):\Op(0)\to \RR^n$$ 
    is a diffeomorphism onto its image. Equivalently, there are smooth functions $y_1,\ldots y_{m-n}:\Op(x\circ\gamma(0))\to \RR^{m-n}$ such that 
    $$(x_j\circ\gamma, \nu_l\circ\gamma)=(x_j\circ\gamma, y(x_j\circ\gamma)).$$
\end{definition}
\begin{lemma}\label{lem:dividedCoords}
    Any $\gamma:\Op(0)\to M$ admits divided coordinates (possibly after shrinking $\Op(0)$). 
\end{lemma}
\begin{proof}
    We find the coordinates by constructing a tubular neighbourhood: we use the parametrization $\gamma$ as coordinates $x$. 
    Choose a normal bundle $\mathfrak n\subseteq TM|_N$ that complements $TN$. Choose a trivialization $\nu:\mathfrak n\to \RR^{m-n}$. Then, the desired coordinates are given by the exponential map with respect to some auxiliary Riemannian metric to a neighbourhood of the zero section
    $$\exp(\gamma(x),\nu_{\gamma(x)}):\Op^n(0)\times \Op^{m-n}(0)\subseteq \RR^n\times\RR^{m-n} \to M.$$
    Note that in these coordinates, $N$ itself is represented by the zero function $y(x)\equiv 0$. 
\end{proof}

It is crucial that if coordinates are divided with respect to some submanifold $N$ which is described by $y$, then they are also divided with respect to every $C^1$-nearby $N'$. Using the diffeomorphism $x\circ\gamma$ as a coordinate change allows us to use $(x_j)$ as coordinates for every nearby $N$ in a coordinated manner. Nearby $N'$ for which the coordinates are also divided are equivalent to functions $y'$ nearby $y$. We have thus proven the following structural statement.

\begin{lemma}
    The bundle of local $n$-dimensional submanifolds with marked point $\Sub(M,n)$ is locally around $\gamma$ modeled by (small) sections of the tubular neighbourhood defined by a normal bundle $\mathfrak n$, together with the choice of marked point. In internal coordinates of $\gamma$ and $\nu$, the model becomes (small) sections together with a choice of marked point of the base $\gamma$ $$C^\infty(\Op^n(x\circ\gamma(0)),\Op^{m-n}(\nu\circ\gamma(0)))\times \Op^n(x\circ\gamma(0))\subseteq C^\infty(\RR^n,\RR^{m-n})\times\RR^n.$$
    This is a bundle over the tubular neighbourhood by the map $(y,x)\mapsto (x, y(x))$.
    The $C^k$-norm topologies on $\Sub(M,n)$ coincide in the local model with the $C^k$-norm topologies on local sections. 
\end{lemma}

\subsubsection{Jets of submanifolds} \label{subsubsec:JetsOfSubmanifolds}
Two nearby submanifolds $N,N'$ whose marked point has coordinates $x\circ\gamma(0)=x\circ\gamma'(0)=x^*$ and who are described in common directed coordinates by the functions $y,y':\Op^n(0)\to\Op^{m-n}(0)$ have a $k$-th order tangency at their marked point if the functions $y,y'$ coincide at $x\circ\gamma(0)$ to $k$-th order $j^ky(0)=j^ky'(0)$. This defines an equivalence relation $\sim$.
\begin{definition}
    The jet bundle of submanifolds is
    $$J^k(M,1) = \sub(M,n)/\sim.$$ 
    It is a bundle over $M$ by sending $[\gamma]$ to $\gamma(0)$. The total space is a manifold. A (local) tubular neighbourhood defines local coordinates of $J^k(M,1)$ modeled on (small) sections of the normal bundle $$J^k(\Op^n(0),\Op^{m-n}(0))\subseteq J^k(\RR^n,\RR^{m-n}).$$
    There is a natural map $j^k:\sub(M,n)\to J^k(M,n)$, which assigns to a marked submanifold its jet at the marked point. 
\end{definition}

We note that the base point map $(y,x)\mapsto (x,y(x))$ coincides with the 0-jet of the function. 

For jet bundles of sections $f:M\to E$, a bundle morphism $F:E\to E'$  naturally induces a bundle morphism $F^k_*:J^k(E)\to J^k(E'), [f]_x^k\mapsto [F\circ f]^k_{x}$. For the jet bundle of embedded submanifolds, this is more delicate: embedded submanifolds might not be mapped to immersions by a smooth map $F:M\to M'$ since a tangent vector $v\in TN$ may be crushed to zero $dF(v)=0$. Generally, we will say that a marked submanifold $N\subset M$ is \emph{compatible} with a smooth map $F:M\to M'$ at its marked point $*\in N$ if $T_xN\cap \ker T_xF=\{0\}$. Then, for a parametrization $\gamma:\Op(0)\to N$ the differential  $dF\circ d\gamma(0)$ is injective, and since injectivity is an open condition, $F\circ \gamma$ is an embedding (after shrinking $\Op(0)$ if necessary).

We denote by $\Sub_F(M,n)$ the set of marked $n$-dimensional submanifolds that are compatible with $F$. On this subspace, $F$ induces a well defined map
$$F_*:\Sub_F(M,n)\mapsto \Sub(M',n);\quad [\gamma]\mapsto [F\circ\gamma].$$
We can transport this notion to jet bundles.
\begin{proposition}
    Let $N,N'$ be two marked sub-manifolds such that $j^1N=j^1N'$. Then, $N$ is compatible with a smooth map $F$ if and only if $N'$ is.
\end{proposition}
\begin{proof}
    The first jet determines the tangent space.
\end{proof}
\begin{corollary}
    For $k\geq 1$, the space of $k$-jets $J^k_F(M,n)$ compatible with $F$ is well defined. Moreover, it is an open subset of $J^k(M,n)$.
\end{corollary}


\subsubsection{One dimensional submanifolds}

We specialize to the setting of one dimensional submanifolds. Then, $\Sub(M,1)$ again modeled on the space 
$$C^\infty(\Op^1(0),\Op^{m-1}(0))\subseteq C^\infty(\RR,\RR^{m-1}).$$
Since there is only one variable, the corresponding jet space $J^k(M,1)$ is modeled on
$$J^k(\Op^1(0),\Op^{m-1}(0))\subseteq J^k(\RR,\RR^{m-1})=\RR\times(\RR^{m-1})^{k+1}$$
and thus has dimension $1+(m-1)(k+1)$.

Two curves intersect  at their marked point if $\gamma^1(0)=\gamma^2(0)=q$. In the local coordinates above with marked point $*$ this becomes a statement about the zero yet: the curves are represented by $y^1,y^2\in C^\infty_V(\Op(*), \RR^{n-1})$ and they intersect if $j^0y^1(*)=j^0y^2(*)$. The crucial observation is that if the two curves intersect \emph{again} $\gamma^1(t_1)=\gamma^2(t_2)$, then the coordinates $(x,\nu)\circ\gamma^1(t_1)$ and $(x,\nu)\circ\gamma^2(t_2)$ must coincide and thus for $* =x\circ\gamma^1(t_1)=x\circ\gamma^2(t_2)$ the values $y^{1}(*)=y^{2}(*)$ must coincide. Our ultimate objective is to bound such an intersection time $\tau$ away from the marked time $*$. 
\begin{definition}\label{def:tangencyorder}
    Two curves $\gamma^1,\gamma^2\in C^\infty_V(\Op(*),\RR^{n-1})$ are \emph{transverse} of order $k$ if they are not tangent of order $k$, i.e., if the $k$-jets at $*$ disagree:
    $$j^ky^{1}(*)\neq j^ky^{2}(*).$$
\end{definition}
Transversality is essential for us, as it prevents accumulation of intersections.
\begin{lemma}\label{lem:transverseQ}
    If two curves $\gamma^1,\gamma^2\in C^\infty_V(\Op(*),\RR^{n-1})$ are transverse of order $k$, then there is an $\varepsilon>0$ such that $y^1(\tau)\neq y^2(\tau)$ for all $\tau\in (-\varepsilon+*,*)\cup(*,*+\varepsilon)$.
\end{lemma}
\begin{proof}
    If the curves do not intersect at their marked time, then the statement follows from continuity. If they do intersect, then the statement follows from the Taylor estimate for $y^{2}-y^{1}$.
    
    It shall be noted that the existence of $\varepsilon$ is independent of the directed coordinate system on $V$, but the size of $\varepsilon$ may degenerate if $dx(\dot\gamma)$ approaches 0. 
\end{proof}

\begin{remark}
If two regular curves coincide $\gamma^1=\gamma^2$, then note that the set of intersection times $\Ical=\{(t_1,t_2)\mid \gamma^1(t_1)=\gamma^2(t_2)$  contains the diagonal as an isolated subset. The lemma above can be applied to $\Ical\setminus \Delta$, i.e., as statement about self-intersections.
\end{remark}

\subsubsection{Vector fields define sections}
Given a smooth flow of a vector field $X:M\to TM$ on a manifold without boundary, not necessarily complete, its integral curves are immersed one-dimensional submanifolds. For a point $x\in M$, there is a neighbourhood $U$ of $x$ in which the flow of $X$ exists for $t\in (-\epsilon,\epsilon)$ for some $\epsilon>0$. Let $\gamma_y:(-\epsilon,\epsilon)\to M$ be the integral curve starting at $y\in U$. The smoothness of $X$ implies that the induced smooth map $j^k\gamma_y:(-\epsilon,\epsilon)\to J^k(M,1)$ extends to a smooth local section $U\to J^k(M,1), x\mapsto j^k\gamma_x(0)$.


\subsection{Hamiltonians}
Here we discuss everything concerning the Hamiltonian vector fields. In Subsubsection~\ref{subsubsec:HamVec}, we will recall how a Hamiltonian vector field is defined. In Subsubsection~\ref{subsubsec:perturbSigma} we explain how we model the perturbation space globally in the level surface $\Sigma$ and how we topologize the perturbations. In Subsubsection~\ref{subsubsec:Covers} we refine the perturbations to covers, which allows us finally in Subsection~\ref{subsec:ContactFormulation} to reformulate the problem from Hamiltonian to locally contact.


\subsubsection{The Hamiltonian vector field}\label{subsubsec:HamVec}

Let $Q^n$ be a manifold. Given local coordinates $q=(q_1,\ldots,q_n)$ on $Q$, we obtain a frame $(dq_1,\ldots,dq_n)$ of $T^*Q$, which defines linear coordinates $(p_1,\ldots, p_n)$ of the fibers $T_q^*Q$: any $p$ decomposes as $p=\sum p_idq_i$. Together, they form coordinates $(p,q)$ on $T^*Q$. 

The form $\omega=\sum dp_i\wedge dq_i$ is symplectic and independent of the coordinates. A Hamiltonian is a function $H:T^*Q\to \RR$ whose purpose is to generate the vector field $X_H$ by solving the Hamiltonian equation 
$$\iota_{X_H}\omega = -dH.$$
With respect to the coordinates $p,q$ the Hamiltonian equation for $X_H$ takes the form 
\begin{align*}
    X_H = &  \quad\sum_i  dH(\partial_{p_i}) \;\partial_{q_i}\\
            &-\sum_i  dH(\partial_{q_i}) \;\partial_{p_i}.
\end{align*}

\begin{lemma}\label{lem:characteristicFlow}
    The vector field $X_H$ is tangent to $\Sigma$ and the conformal class of the vector field $X_H|_\Sigma$ depends only on $\Sigma$. More precisely, if $H,G\in\Ham_s$ and $\Sigma=H^{-1}(0)=G^{-1}(0)$, then there is a smooth function $f:\Sigma\to\RR\setminus \{0\}$ such that $X_G|_\Sigma=fX_H|_\Sigma.$
\end{lemma}
\begin{proof}
    Since $dH(X_H)=-\omega(X_H,X_H)=0$, the vector field $X_H$ preserves $H$ and is thus tangent to the level set $X_H\in T\Sigma$. On $\Sigma$, $\omega$ is a maximally non-degenerate 2-form, which implies that $\chi_\Sigma=\ker\omega|_\Sigma=\RR\cdot X_H\subseteq T\Sigma$ is a line field (which is called the characteristic line field). For any other vector $V\in T\Sigma$ we have $\omega(X_H,V)=-dH(V)=0$, thus the vector field $X_H$ is a section of $\chi$. Since $\Sigma$ is a regular level set of $H,G$, the vector fields don't vanish $X_H|_\Sigma,X_G|_\Sigma\neq 0$ and thus we can find the desired function $f$. More precisely, the function $f$ satisfies $dG|_\Sigma=fdH|_\Sigma$.
\end{proof}

Given $H$, this defines a flow line $\varphi^t_{X_H}(p,q)$ for each point $(p,q)$ that is marked at $t=0$. Thus, a Hamiltonian $H$ immediately produces a section
$$\sigma:\Sigma\mapsto \Sub(\Sigma,1);\quad (p,q)\mapsto \varphi^t_{X_H}(p,q)$$
of the bundle $\Sub(\Sigma,1)$ and the corresponding jet bundles
$$\sigma_k:\Sigma\mapsto J^k(\Sigma,1);\quad (p,q)\mapsto j^ky(t),$$
where for the local expression the functions $y$ are extracted from any local coordinates $(p,q)$ such that $dq_1(X_H)\neq 0$, which is possible since $\Sigma\to Q$ is a submersion. From Lemma~\ref{lem:characteristicFlow} we immediately obtain the following.
\begin{corollary}
    The sections $\sigma$ and $\sigma_k$ only depend on $\Sigma$, not on the choice of $H\in \Ham_s$ such that $H^{-1}(0)=\Sigma$.
\end{corollary}

We can define an equivalence relation on $\Ham_s$ where $H\sim G$ if $\Sigma_H=\Sigma_G$. Note that if $g$ is a nonzero function on $T^*Q$, then $gH\sim H$. The following Lemma shows that the converse is also true.
\begin{lemma}\label{lem:Hadamard}
    Let $H\sim G$ in $\Ham_s$. Then, there is a smooth positive function $g$ such that $G=gH$. This function depends smoothly on $H$ and $G$.
\end{lemma}
\begin{proof}
    Away from the zero set, we can set $g=G/H$, which smoothly depends on $H$ and $G$. This function $g$ extends smoothly over the zero set by Hadamards lemma.
\end{proof}
Thus, it makes sense to say that the space of hypersurfaces is a quotient space $\Ham_s/C^\infty(T^*Q,\RR_{\neq 0})$. We do not make use of this directly, instead we will construct a perturbation space for the set of hypersurfaces and show that this perturbation space times positive functions is $\Ham_s$, see Corollary~\ref{cor:HamiltoniansToSurface}.

\subsubsection{Perturbations of $\Sigma$}\label{subsubsec:perturbSigma}
Hypersurfaces are submanifolds and thus a neighbourhood in the set of hypersurfaces can be described as graphs of a normal bundle.

It would be tempting to believe that $C^k$-graphs in these normal bundles are the same $C^k$ perturbations of $H$. However, since we deal also with non-compact $\Sigma$, this is not true, as the following example shows. 

\begin{example}\label{ex:InfiniteSpheres}
    Consider a bump function $b:[-1,1]\times \RR \to\RR$ with support in a small ball $B_{0.1}(0,0)$ and such that $b^{-1}(1)$ is a transversely cut out sphere. Periodically (in $x$) extend $b$ to $\tilde b$ and let $B(x,y)=(\tilde b(x,y)-1)\cdot e^{-x^2}$. Then, $B^{-1}(0)$ is a countable union of spheres. For $\varepsilon\to 0$, $B+\varepsilon$ converges to $B$ in $C^\infty$. However, for every $\varepsilon>0$ the level set $(B-\varepsilon)^{-1}(0)$ is only a finite union of spheres. However, we can exhaust $\RR^2$ by compact sets, say balls $B_N(0,0)$ for $N\in \NN$. For each of these compact sets, there is an $\varepsilon$ such that $(B|_{B_N(0,0)}-\varepsilon)^{-1}(0)$ is a graphlike perturbation of $(B|_{B_N(0,0)})^{-1}(0)$. 
\end{example}

We thus pick a relatively compact subset $K\subseteq T^*Q$, and we assume it is open and has smooth boundary. In the following, all objects are restricted to $K$: $\Sigma\cap K$, $H|_K$ etc., but for the sake of readability we only mention this when it makes a difference. We explain further how we patch neighbourhoods together in Subsubsection~\ref{subsubsec:Covers}.

Pick any line bundle $TT^*Q\supset\mathfrak n\to \Sigma$ that complements $T\Sigma$. A good candidate comes from an auxiliary Riemannian metric on the base space $Q$, which together with the dual metric $g^*$ defines a Riemannian metric $g\oplus g^*$ on $T^*Q$. Then, $\mathfrak n=T\Sigma^\perp$ is a natural choice. The line bundle is trivial because $\Sigma$ is a regular level set of $H$ and thus $dH|_{\mathfrak n}$ is a choice of volume form. Choose a trivialization $\nu:\mathfrak n\to \RR$. 

A neighbourhood of $K\cap\Sigma$ is parametrized by the exponential map
$$\exp: K\cap\Sigma\times \Op(0)\subseteq \mathfrak n \to T^*Q;\quad (x,n)\mapsto \exp(x)[\nu^{-1}_x(n)].$$
\begin{remark}
    $K$ is important here so that we can uniformly choose $\Op(0)$ for $\nu$. If we make $K$ larger, then we may to shrink $\Op(0)$ to avoid overlaps. 
\end{remark}

We interpret $\Sigma$ as the graph of the zero section. A section $f$ of $\mathfrak n$ produces a Hamiltonian of the form $$H_f=\nu\circ(\exp^{-1})-f.$$ 
On the other hand, a $C^1$ close $H'$ will produce a $C^1$-close graphlike hypersurface $K\cap\Sigma'$ by the implicit function theorem.
\begin{remark}
    $K$ is important here because the implicit function theorem may fail on non-compact sets, as illustrated by Example~\ref{ex:InfiniteSpheres}.
\end{remark}
Thus, there are maps between sections of $\mathfrak n$ and $\Ham_s$ defined in neighbourhoods of the zero section and $H$, respectively
$$H_f:C^\infty(\mathfrak n)\supset \Op(0) \longrightarrow \Ham_s, \quad \Xi:\Ham_s\supset \Op(H)\longrightarrow C^\infty(\mathfrak n).$$
We have that $\Xi$ is a left-inverse to $H_f$, $\Xi\circ H_f=id$ which is really just saying that the radial Hamiltonian produced by graph set has that graph as zero set. The map $H_f$ is not surjective onto a neighbourhood of $H$ as there is no reason for elements of $\Ham_s$ to be homogeneous. However, we can multiply $H_f$ with any nonzero function $g$ to produce the same hypersurface. Let the compact neighbourhood $K$ be described by sections of the normal bundle. We define the mapping 
$$\Fcal: C^\infty(\mathfrak n)\times C^\infty(K,\RR\setminus\{0\}) \supset \Op(0,1) \mapsto \Ham_s; \quad (f,g)\mapsto gH_f.$$

\begin{corollary}\label{cor:HamiltoniansToSurface}
    Let $K$ be a compact neighbourhood of $K\cap\Sigma$, modeled by a small enough neighbourhood of the zero section of the bundle $\mathfrak n$. Let $\Ucal\subseteq C^\infty(\mathfrak n)\times C^\infty(\mathfrak n,\RR\setminus\{0\})$ be a small enough $C^k$ neighbourhood of $(0,1)$ such that all images are contained in the neighbourhood described by the normal bundle. We denote $\Vcal=\Fcal(\Ucal)$. Then $\Fcal:\Ucal\to\Vcal$ is a homeomorphism onto its image with respect to the Whitney $C^k$ topology (weak or strong). 
    
    In particular, if $\Acal\subset \Wcal =\pi_{C^\infty(\mathfrak n)}\Ucal$ is a $C^k$-residual subset of $\Wcal$, then $F(\Acal \times C^\infty(\mathfrak n,\RR\setminus\{0\}\cap \Ucal)$ is $C^k$-residual in $\Vcal$.
\end{corollary}
\begin{proof}
    The first claim is a local version of Lemma~\ref{lem:Hadamard}. The second claim is obvious.
\end{proof}

\subsubsection{Refinement by open covers}\label{subsubsec:Covers}
In the previous section, we restricted our attention to compact subsets $K\subseteq T^*Q$ in order to employ the implicit function theorem. This is not the only reason to restrict to a cover of $\Sigma$: since we are free to choose $H$, we can as well choose a nice one. We do so by choosing the normal bundle $\mathfrak n$ fiberwise radial with respect to some center. We explain here how to define bundles and how to switch between local patches. In Subsection~\ref{subsec:ContactFormulation} we explain why this makes $H$ particularly useful.

A set $K\subseteq T^*Q$ is \emph{local} if there are local coordinates $q_i$ in $Q$ such that $\pi_QK$ lies in the domain of the coordinates. As usual, $dq_i$ determines a frame in $T^*Q$, the coordinates of which we call $p_i$. 
\begin{definition}
    A local set $K$ intersecting $\Sigma$ is \emph{radial} if there is a section $P(q)=\sum P_i\partial dq_i$ of $T^*Q$ that is constant in $q$ with respect to the frame $dq_i$ such that every linear ray emanating from $P$ in $T^*_qQ$ interesects $\Sigma\cap K$ at most once, and all such intersections are transverse.
\end{definition}
    The transversality condition means that the fiberwise radial vector field $Y_P=(p-P)\partial_{p-P}$ emanating from $P$ is transverse to $\Sigma\cap K$. Equivalently, the 1-form 
    $$\lambda_{P}= \sum(p_i-P_i) dq_i = (p-P) dq$$
    does not vanish on $\Sigma$. 
 
In such a radial chart we define $\mathfrak n$ to be the radial line bundle 
$$\mathfrak n(p,q)= \operatorname{span}(p,P(q)).$$ 
We choose on $\mathfrak n$ the trivialization $\nu:\mathfrak n\to\RR$ such that $\nu(P(q))=-1$. This naturally defines a radial function $r=\nu+1$. The Hamiltonian 
$$H(p,q)=H(r)=r-1$$
then has $\Sigma$ as a hypersurface. 

\begin{lemma}
    $T^*Q$ admits a countable open cover $\{K_i\}_{i\in\NN}$ such that each $K_i$ either does not intersect $\Sigma$ or is radial. We call such a cover radial. 
\end{lemma}
\begin{proof}
    For every point $p\in\Sigma_q$ and for any $P\in T^*_qQ$ not contained in the hyperplane tangent to $\Sigma_q$ at $p$ the radial line from $P$ to $p$ intersects $\Sigma$ transversely. Extend $P$ constantly. Then, the radial line from $P(q)$ to any point in a small enough neighbourhood $p\in U\subseteq \Sigma$ intersects $\Sigma$ transversely. The set $\{(x,r)\in \mathfrak n\mid x\in U, r\in(1-\varepsilon,1+\varepsilon)\subseteq\RR\}$ is then a relatively compact neighbourhood of $(p,q)$. We can thus cover $\Sigma$ by such sets. We add an open cover of $T^*Q\setminus \Sigma$ and choose an open subcover.
\end{proof}

Such a cover gives rise to the notion of local properties, and we can piece them together to generic global properties. For a radial set $K$ with respect to a hypersurface $\Sigma$, we define $\Ham_s(\Sigma,K)$ the set of Hamiltonians whose zero set are graph-like in $K$. 
\begin{definition}\label{def:LocalProperty}
    A property $\Acal_K\subseteq \Ham_s$ is local in $K$ if either 
    \begin{itemize}
        \item for any $H\notin\Ham_s(\Sigma,K)$, $H\in\Acal_K$,
        \item for any $H\in\Acal_K$ and $H|_K=H'|_K$ then $H'\in \Acal_K$.
    \end{itemize}    
\end{definition}
To puzzle these local properties together, we need to make sure we describe locally all of $\Ham_s$. For this, choose a dense sequence $H_i$ in $\Ham_s$. For each $H_i$ and for each $n$ find a countable radial cover by relatively compact sets $K_{i,j,n}$ such that the diameter of $K_{i,j,n}$ is smaller than $\frac 1n$.  
\begin{lemma}\label{lem:ReduceToCover}
    Let $\{\Acal_{i,j,n}\}$ be a collection of properties that are local in $K_{i,j,n}$. If each $\Acal_{i,j,n}$ is $C^k$-residual, then the total property 
    $$\Acal=\bigcap \Acal_{i,j,n}$$
    is $C^k$-residual.
\end{lemma}
\begin{proof}
    Countable intersections of residual sets are residual. 
\end{proof}
The point of this formulation is that every $H\in \Ham_s$ and every point $x\in \Sigma_H$ a subsequence $H_{i_\iota}\to H$ approximates $H$ and for each $n$ and $\iota$ large enough there is a $j_\iota$ such that  $x\in K_{i_\iota,j_\iota,n}$. Around $x$, $\Sigma_H$ is well approximated by its tangent plane and so for $n$ large enough,  $K_{i_\iota,j_\iota,n}$ is radial for $H$ and thus we will be able to investigate $H$ by the property $\Acal_{i_\iota,j_\iota,n}$ using the construction below. It is important to note that we consider properties that only depend of the functions $f$.

We are left with declaring how we treat perturbations locally. Let $H$ be a Hamiltonian generating $\Sigma$ and choose a countable radial cover $\{K\}$.
Choosing a sub-division of the radial cover, we can arrange such that for each $K$, the union
$$K^*:=\bigcup_{K\cap K'\neq \emptyset} K'$$
is still radial. We choose cutoff functions $\chi_K$ such that $\chi_K|_K=1$, $\operatorname{Supp}(\chi)\subseteq K^*$. Let us describe how we construct from $f$ a global Hamiltonian. Let $H_{K^*}$ be the radial Hamiltonian $H_{K^*}=r-1$ that is only defined in $K^*$. Then, $G=H_{K^*}/H$ is a smooth nonzero function in $K^*$. Furthermore, $\chi_KG$ has support in the interior of $K^*$ and we can extend it smoothly by 0. We define the modified Hamiltonian 
$$\overline{H_K}=(\chi_K G) H + (1-\chi_K) H.$$
This has the same zero set as $H$ and coincides with the radial Hamiltonian $H_K$ in $K$. 

For a perturbation $f:\Sigma\cap K\to \RR$, we note that there is a smooth extension $f^*:\Sigma\cap K'\to \RR$ since we assumed $K$ to have smooth boundary. So, given $H$ we first normalize $H$ to $\overline{H_K}$ and then perturb by $f$ as before using $G_f= (f^*r-1)|_{K^*}/H$
$$\overline{H_f}=(\chi_K G_f) H + (1-\chi_K) H.$$

For $f$ $C^k$-close to 1, we can choose a $C^k$-small extension $f^*$ and as a result, $\overline{H_f}$ will be $C^k$-close to $\overline{H_K}$. We note that $\overline{H_f}|_K= fr-1$ only depends on $f$, not on our fixed choice of $H$. 
\begin{definition}\label{def:radialProperty}
    Let $\Bcal_K(\Sigma\cap K,\RR)$ be a property of functions. It induces \emph{radial} property $\Acal_K\subseteq \Ham_s$ by
    \begin{itemize}
        \item For any $H\notin\Ham_s(\Sigma,K)$, $H\in\Acal_K$.
        \item For any $H\in\Acal_K$ we have $H|_K= G (\overline{H_f})|_K$ for some $f$. If $f\in \Bcal_K$, then $H\in\Acal_K$
    \end{itemize}    
\end{definition}
\begin{lemma}
    Properties that are radial in $K$ are local in $K$. 
\end{lemma}
\begin{proof}
    This is true by definition.
\end{proof}
Collectively, these statements imply that our desired conclusion holds for residual Hamiltonians if we manage to prove that the desired conclusion holds locally in a radial chart $K$ for generic functions $f$.

\subsection{Local contact formulation}\label{subsec:ContactFormulation}

The radial neighbourhoods from the previous section are particularly nice: the radial (Liouville) vector fields $Y_P$ and the (Liouville) 1-form $\lambda_P$ interact with the symplectic form in the following ways:
\begin{align*}
    \iota_{Y_P}\omega &=\lambda_P,\\
    d\lambda_P&=\omega,\\
    \Lcal_{Y_P}\omega &=\omega.
\end{align*}
If for $p_0\in \Sigma_q$ with $Y_P\pitchfork \Sigma$, choose a neighbourhood $(p_0,q_0)\in U\subseteq \Sigma$ such that $U\times I_\varepsilon$ parametrizes a radial neighbourhood of $U$, where $I_\varepsilon= (1-\varepsilon,1+\varepsilon)$. 

The restriction $\alpha=\lambda_P|_U$ is then a contact form and the characteristic flow of $\Sigma$ is in $U$ a reparametrisation of the Reeb flow of $\alpha$, defined by the Reeb vector field $R$ which solves the equations
\begin{align*}
    \alpha(R)&=1,\\
    \iota_Rd\alpha &=0.
\end{align*}
The Reeb vector field of $\alpha$ coincides with the Hamiltonian vector field of $H=r-1$.

\begin{remark}    
If $\Sigma$ is fiberwise starshaped with respect to the $0$ section $P=0$, then the tautological 1-form $\lambda=\lambda_0=pdq$ is a global choice of contact form. 
\end{remark}

\subsubsection{From perturbing the surface to perturbing the flow}

Perturbing surfaces is less convenient than perturbing the flow on a fixed surface. We achieve this in the following way: Projecting along our normal bundle $\nu$, we can get a map $\pi_\Sigma$ that maps the total bundle of $\nu$ to $\Sigma$. Restricted to a nearby hypersurface $\pi_\Sigma|_{\Sigma'}$, this map is a diffeomorphism. Then, studying the flow of $X_{H'}$ on $\Sigma'$ is the same as studying the flow of $D\pi_\Sigma|_{\Sigma'}(X_{H'})$ on $\Sigma$.

This interacts particularly well with the contact description of the Hamiltonian flow, using a local neighbourhood $U$ and our local choice of constant section $P$. Then, $\pi_U$ is the radial projection with respect to the fiberwise center $P$ that identifies the neighbourhoods
$$\pi_U:U'\to U.$$
Further, since $\lambda_P(q,s(p-P)+P) = s\lambda_P(q,p)$, the kernels agree and $\pi_U$ is actually a contactomorphism. 

The characteristic flow on $U'$ is parallel to the Reeb flow of $\lambda_P|_{U'}$. Since $\lambda_P$ is dilation-equivariant and ${U'}=fU$, we have that the projection of the characteristic flow $D\pi_U X_H$ is parallel to the Reeb flow of $$(\pi_U^{-1})^*\lambda_P|_{U'} = f\lambda_P|_U= f\alpha.$$ 
Thus, we can say that locally around a point in $\Sigma$, the following three views on perturbations are the same
$$\{\mbox{Nearby hypersurfaces}\}\cong\{\mbox{Nearby contact forms on }U\}\cong C^\infty(U,I_\varepsilon),$$
and thus all three are describing perturbations of $H$ up to reparametrization of the characteristic flow. We topologize the set of perturbations as a function space with the $C^k$ topologies. For $k\geq 1$, these topologies coincide with the $C^k$ topology of $\Ham_s$ pulled back by the map $H_f$. 

We can finish by recalling that the Reeb flow of $f\alpha$ is the contact Hamiltonian flow of $h=\frac 1f$ with respect to $\alpha$, i.e., the vector field $R_h$ cf.~\cite{Geiges}(Chapter 2.3),
$$\alpha(R_h)=h,\quad \iota_{\pi_\xi R_h} d\alpha = dh(R)\alpha - dh.$$

\begin{proposition}
    For a radial relatively compact set $K$ with $K\cap\Sigma=U$, let a property $\Bcal_K\subseteq C^\infty(U,\RR\geq 0)$ be a property of contact Hamiltonians $h$. By $h\mapsto \frac 1h=f$ it induces a radial property $\Acal_K$ by Definition~\ref{def:radialProperty}.
\end{proposition}

\begin{remark}
    We do not assume the hypersurface to be of contact type, and we do not exploit any stable Hamiltonian or framed Hamiltonian structure, as all these notions are global information of the hypersurface. We only note that we have in a local chart a local Liouville description of the flow, which is weaker than the above notions. 
\end{remark}


\section{Proof of the main theorem}\label{subsec:Bas}

We will find iterative conditions for accumulated intersection and break them: A zeroth order condition is that the curves intersect, a first order condition that they are tangent etc.


\subsection{The homopodal relations}

Ultimately, we are interested in intersections of flow lines of the characteristic flow in $\Sigma$ \emph{after} projection to the base manifold $Q$. Recall from Section~\ref{subsubsec:JetsOfSubmanifolds} that we denote by $\Sub_{\pi_Q}(\Sigma,1)$ the set of submanifolds compatible with the projection, so that $\pi_Q\colon \Sigma\to Q$ induces a map
$$\pi_Q:\Sub_{\pi_Q}(\Sigma,1)\to\Sub(Q,1).$$
In this 1-dimensional context compatible regular curves are those that project to regular curves in the base. This is the same as to demand that they are non-vertical $D\pi_Q\dot\gamma\neq 0$. Similarly, we find induced maps for the jet bundles 
$$\pi_Q^k:J^k_{\pi_Q}(\Sigma,1)\to J^k(Q,1).$$
\begin{lemma}\label{lem:submersiveHam}
    For $H\in\Ham_s$, $d\pi_QX_H(x)\neq 0$ for $x\in\Sigma$.
\end{lemma} 
\begin{proof}    
Since $\pi_Q:\Sigma\to Q$ is a submersion, there is a lift from $TM$ to $\ker dH\subseteq TT^*Q$. So $\operatorname{Span}(\ker dH, TT_q^*Q)=TT^*Q$. If $dH$ would also vanish on all vertical vectors $dH(TT^*_qQ)=0$, then we would conclude that $dH=0$. But then $\Sigma$ would not be a regular level set of $H$. 
\end{proof}
As a result, the sections $\sigma$ and $\sigma_k$ of the submanifold bundles are compatible with the induced maps $\pi_Q$:
\begin{align*}
    \sigma\colon &\Sigma\to \Sub_{\pi_Q}(\Sigma,1) \stackrel{\pi_Q}\to\Sub(Q,1),\\
    \sigma_k\colon &\Sigma \to J^k_{\pi_Q}(\Sigma,1)\stackrel{\pi_Q^k}\to J^k(Q,1).
\end{align*}

\begin{remark}
    The map $\pi_Q^1$ is simply the differential $D\pi_Q$ at a point in $\Sigma$. This expression makes sense at all of $T^*Q$, not only on $\Sigma$. It shall be noted that this coincides with the Legendre transform: the vertical differential $d_vH=d(H|_{T^*_qQ})$ is an element of $T^*T^*_qQ$. Since $T^*_qQ$ is a vector space, it is canonically identified with its tangent space and thus $T^*T^*_qQ=T^{**}_qQ$, which can be canonically identified with $T_qQ$. Thus, $d_vH\in T_qQ$. It coincides with $d_vH=d\pi_Q^1X_H$, which is easily visible from the Hamiltonian equations in classical form 
    \begin{align*}
         X_H = &  \quad\sum_i  dH(\partial_{p_i}) \;\partial_{q_i}\\
                &-\sum_i  dH(\partial_{q_i}) \;\partial_{p_i}.
        \end{align*} 
    One can also set $d_vH=\Lcal_H:T^*Q\to TQ$. If $H$ is fiberwise strictly convex and superlinear at infinity, then $\Lcal_H:T^*Q\to TQ$ is a diffeomorphism and $F:=H\circ (\Lcal_H)^{-1}$ is the classical Legendre transform of $H$ (and then also $H= F\circ (\Lcal_F)^{-1})$. 
\end{remark}


\subsubsection{Homopodal relations}
We are interested in pairs of points $(q_i,p_i)\in\Sigma$ such that their flow lines have projections to $Q$ that are tangent to $k$-th order
$$\pi_Q^k\circ \sigma^k(q,p)(q_1,p_1)=\pi_Q^k\circ \sigma^k(q,p)(q_2,p_2).$$
For $k=0$ this just indicates an intersection. For $k=1$, this indicates a tangency. If $\Sigma$ is fiberwise convex (as in the Finsler case), then for each point there are only two points whose flow lines are tangent: The point itself and its opposite, which is traditionally called antipode. This inspires us extend the terminology: homopodes are the tangencies, isopodes are the equally oriented tangencies and antipodes the differently oriented tangencies. 

\begin{definition}
    The \emph{homopodes} of order $k$ of $(q,p)$ are the points $(p',q')\in\Sigma$
    \begin{align*}
         \Hcal^k(p,q)&=\{(q',p')\mid \pi_Q^k\circ \sigma^k(q_1,p_1)=\pi_Q^k\circ \sigma^k(q_2,p_2)\}\subseteq\Sigma,\\
         \Hcal^k(\Sigma)&=\{(q,p,q',p')\mid (q',p')\in \Hcal^k(q,p)\}\subseteq \Sigma\times\Sigma.
    \end{align*} 
    For $k>0$, we can divide the homopodes further into the \emph{iso-, and antipodes} $\Hcal^k_\pm(p,q)$ and $\Hcal^k_\pm(\Sigma)$ by demanding that $D\pi_QX_H(p,q)=sD\pi_QX_H(p',q')$ for $s>0$ and $s<0$, respectively.
\end{definition}

This definition fits into the following diagram where $\Hcal^k$ is the preimage of the diagonal of $\Delta_{J^k(Q,1)}$ in $\Sigma\times\Sigma$.

\begin{equation}\label{eq:Diagram}    
  \begin{tikzcd}
    &\Sub_{\pi_Q}(\Sigma,1)\times\Sub_{\pi_Q}(\Sigma,1) \arrow{r}{\pi_Q\times\pi_Q}\arrow{d}{j^k\times j^k} &\Sub(Q,1)\times\Sub(Q,1)\arrow{d}{j^k\times j^k}  \\ 
    \Sigma\times\Sigma \arrow{ur}{\sigma\times\sigma} \arrow{r}{\sigma^k\times\sigma^k} & J^k_{\pi_Q}(\Sigma,1)\times J^k_{\pi_Q}(\Sigma,1) \arrow{r}{\pi^k_Q\times\pi^k_Q}& J^k(Q,1)\times J^k(Q,1)\\
     \Hcal^k(\Sigma) \ar[u,hook] \arrow{rr} & & \Delta_{J^k(Q,1)} \ar[u,hook]
  \end{tikzcd}
\end{equation}
The sets $\Hcal^k_\pm(\Sigma)$ are unions of connected components of $\Hcal^k(\Sigma)$, but they are not visible in the diagram since the right side column is oblivious to parametrization.

\subsubsection{Example of the heart}
Let us investigate an example in $Q=\RR^2$. We look at a specific $\Sigma$ and intend to see that $\Hcal^1(\Sigma)$ is mostly a manifold. 
\begin{example}\label{ex:heart}
    The Hamiltonian $H(p,q)$ in $T^*\RR^2$ we consider should only depend on $p$ and be such that every fiber of $\Sigma$ has a heart shape as in Figure~\ref{fig:heart}. Since the fiber $\Sigma_q$ is a circle, the product $\Sigma_q\times\Sigma_q$ is thus a torus. As depicted, the interesting points are the inflection points of the heart and their parallels. As visible in the example, the antipodal set is a manifold, in this case a circle that winds off-diagonal around the torus. It has horizontal and vertical tangents at points whose coordinates are combinations of inflection points and their parallels. For the isopodal set the situation is similar: It the union of the diagonal and a circle with horizontal and vertical tangents as for the antipodal set. However, the intersection points of diagonal and this circle are not manifold points.

    Because of the $q-$invariance, the homopodal set extends to the set $\Delta_{\Sigma}\times \Hcal^1_{q,q}$. 
\end{example}
    \begin{figure}
        \centering
        \includegraphics[width=0.75\linewidth]{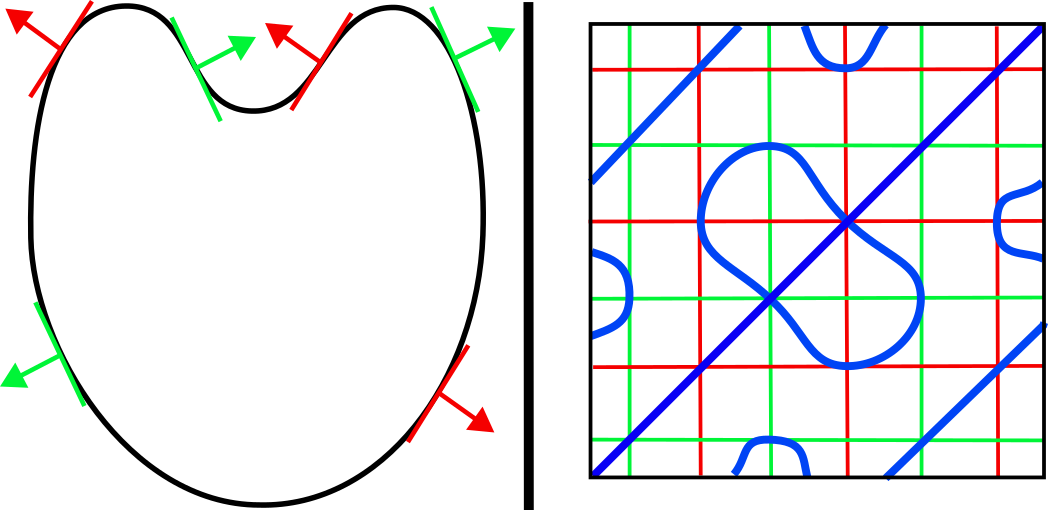}
        \caption{On the left hand side we see a fiber $\Sigma_q$ with six marked points: The inflection points and their parallels in green and red, respectively. On the right hand side we see the torus $\Sigma_q\times\Sigma_q$ with the coordinate lines of the red and green points. The blue lines are the fiber $\Hcal^1_{q,q}$. Note that it consists of two connected components: An off-diagonal circle representing the antipodals which curves at the inflection points times their parallels, and a contractible circle united with the diagonal representing the isopodals.}
        \label{fig:heart}
    \end{figure}

As Example~\ref{ex:heart} shows, the set $\Hcal^1(\Sigma)$ can have non-manifold points that cannot be perturbed away. However, these points are concentrated around the inflection points of the Hamiltonian. To describe the inflection points better, we consider the second vertical derivative $d_v^2H$ of $H$. In coordinates, this can be understood as the matrix $(\partial_{p_i}\partial_{p_j}H)_{ij}$. We are in particular interested how it behaves along $T\Sigma$. If $\ker d_v^2H\cap T\Sigma(p,q)$ is non-degenerate, then the fiber restriction $\pi^1_Q\circ\sigma^1|_{\Sigma_q}$, i.e., the map that assigns to a momentum its direction, is a submersion. However, if $\ker d_v^2H\cap T\Sigma(p,q)$ is at least 1-dimensional, then the fiber restriction fails to be a submersion at $(p,q,p,q)$. In Example~\ref{ex:heart}, these are the inflection points. 

\begin{definition}\label{def:inflectionPoints}
    The set of order $k$ inflection points is
    $$\Sigma^k=\{(q,p)\in \Sigma\mid \dim\ker d_v^2H\cap T\Sigma(p,q)\geq k\}.$$
\end{definition}

Note that this set is independent of the choice of regular $H$ generating $\Sigma$: the kernel of $d^2_vH\cap T\Sigma$ is the subspace where $\Sigma$ has a second order tangency to its tangent plane, which is independent of $H$. Away from the inflection points, because of non-degeneracy of $d^2_vH|_\Sigma$, there is a neighbourhood $U(p)\subseteq \Sigma$ on which 
$$\pi^1_Q\circ \sigma^1|_{U(p)}:U(p)\to J^1(Q,1)$$ 
is a local diffeomorphism. 

If for a homopodal pair $(p_1,q_1,p_2,q_2)$ one of the points is non-degenerate, say $(p_2,q_2)\notin\Sigma_1$, then the set $\Hcal(\Sigma)$ is a $(2n-1)$ dimensional manifold in a neighbourhood of $(p_1,q_1,p_2,q_2)$, parametrized by the graph $(p,q, (\pi^1_Q\circ \sigma|_{U(p_2)})^{-1}\circ\pi^1_Q\circ \sigma(p,q))$ with $(p,q)$ in a neighbourhood of $(p_1,q_1)$. 
There are of course pairs where both points are inflection points. In this simple example this only happens for homopodal points if the two inflection points are the same, which means it is a diagonal point (which we do not care about in this paper). In more complicated examples, such inflection-inflection pairs also appear off the diagonal. We will show that generically, these are also manifold points of $\Hcal(\Sigma)$.

\subsubsection{Jets of Hamiltonians} 
We now turn to the central ingredient of the proof: we need to perturb $H$ so that the pullback of $\Delta_{J^k(Q,1)}$ in Diagram~\ref{eq:Diagram} is a manifold. In order to do so, we make sure that the mechanism that constructs from the Hamiltonian $H$ the section of the manifold jets $\sigma^k$ factors through the jet of Hamiltonians $j^{k+1}H$. 

\[
  \begin{tikzcd}
    & \Gamma(J^{k+1}(\Sigma,\RR_{>0}))\arrow{dr} &  \\ 
    C^{k+1}(\Sigma,\RR_{>0}) \arrow{ur}{j^{k+1}}\arrow{rr}{\sigma^k} & & J^k_{\pi_Q}(\Sigma,1)
  \end{tikzcd}
\]

We do so in a local contact description so that we can use the contact Hamiltonian equations. Let $(U,\alpha)$ be a Domain in $\RR^n$ with contact form $\alpha$ (later this will be a chart in $\Sigma$). We first note that the contact Hamiltonian vector field as a map 
$$C^{k+1}(U,\RR_{>0})\to \Gamma^{k}(TU);\; h\mapsto R_h$$ 
decreases the regularity by 1. We also note that - since we assume $h>0$ - the resulting vector field is positively transverse to $\xi$. We denote by $J_{\xi}^k(U,1)$ the jets of submanifolds that are transverse to $\xi$. Note that in the cotangent bundle $TT^*_qQ\cap T\Sigma\subseteq \xi$ and therefore $J_{\xi}^k(U,1)\subseteq J_{\pi_Q}^k(U,1)$.  
\begin{definition}\label{def:jetcontacthamiltonian}
    The assignment of contact Hamiltonian vector field $h\mapsto R_h$ induces a map 
    $$\Rcal_k\colon J^{k+1}(U,\RR_{>0}) \longrightarrow J_\xi^k(U,1)\subseteq J_{\pi_Q}^k(U,1),\quad G\mapsto \sigma^k_h(h)(\pi_U(G)),$$
    where $G$ is a jet, $h$ is a positive function with $j^k(h)=G$ and $\pi_U(G)$ are the coordinates of $G$ in $U$.
\end{definition}
This map is well defined as a different choice $h'$ with the same jet $G$ results in a vector field $R_{h'}$ that coincides with $R_h$ to order $k$, and thus the $k$-jet of sub-manifolds coincide. 

\begin{theorem}\label{thm:R_G}
    The jet contact Hamiltonian vector field map $\Rcal_k\colon J^{k+1}(U,\RR_{>0}) \longrightarrow J_*^k(U,1)$ is a surjective map of bundles over $U$. Further, it is differentiable and submersive.
\end{theorem}
\begin{proof}
    It is clear that $\Rcal_k$ preserves the fibers of the bundles, and that $\Rcal_k$ is differentiable.\\    
    \bparagraph{Surjectivity} Let $\Psi_0\in  J_*^k(U,1)$ be a jet that is transverse to the contact structure. We will need to make choices about the following:
    \begin{itemize}
        \item choose a parametrized path $\psi_0(t)$ that represents $\Psi_0$,
        \item construct a contact Hamiltonian $h_0$ that has this parametrization as a flow line.
    \end{itemize}
    The jet $j^{k+1}h$ will then be our pre-image. Anticipating the submersive step, we make sure that the choices can be made in a uniform way for a neighbourhood of $\Psi_0$.
    
    Pick a representative path $\psi$ of $\Psi_0$. We first choose coordinates that are directed with respect to $\psi|_{\Op(0)}$. I.e., we choose $(x_1,\ldots,x_m)$ such that $x_1$ is monotone along any short enough representative of $\Psi_0$, which means $dx_1|_{\psi_0}\neq 0$. For any $\Psi$ near $\Psi_0$ and any representative $\psi$ of $\Psi$, this coordinate system is also directed for $\psi$, because $dx_1|_{\psi}\neq 0$ is an open condition (possibly we need to shrink $\Op(0)$). Reparametrizing $t$ by $x_1$ and noting that the marked time of $\Psi$ is $x_1(\Psi)$, we can write 
    \begin{itemize}
        \item $x_1(\psi(t))=t$,
        \item $j^k\psi(x_1(\Psi))=\Psi$.
    \end{itemize}
    A valid choice for representatives $\psi$ of $\Psi$ are polynomials in the coordinates $x_2,\ldots,x_m$ of degree $\leq k$. Then, the assignment $\Psi\mapsto\psi$ is uniquely determined and smooth in $\Psi$. 
    
    To choose a contact Hamiltonian $h$, we use machinery from contact geometry. We observe that a point is an isotropic submanifold of $\Sigma$, and that $\psi$ is thus an isotopy $\psi(t)$ of isotropic submanifolds. We employ the Isotropic Isotopy Extension Theorem~\cite[Theorem 2.6.2]{Geiges} to produce a (time dependent) Hamiltonian $h_t$ generating a contact isotopy $\varphi(t)$ (we choose the start time as $\varphi(x_1(\Psi))=id$) whose restriction to the point $\pi_{J^0_*(U,1)}\Psi$ is the isotropic isotopy $\varphi(t)|_{\pi_{J^0_*(U,1)}\Psi}=\psi(t)$. Geiges' proof first finds conditions on the 1-jet of $h$ along $\psi(t)$ and then extends linearly. The 1-jet conditions on $h$ are (note that the right hand sides are determined by $\psi$: 
    \begin{align*}
         h_t(\psi(t))&=\alpha(\dot\psi(t)) \\ &\mbox{This is forced by the first contact Hamiltonian equation.}\\
         dh_t(\psi(t))|_\xi &= -\iota_{\dot\psi(t)}d\alpha|_\xi, \\& \mbox{This is forced by the second contact Hamiltonian equation along $\xi$}.\\
         dh_t(\psi(t))[\dot\psi(t)]  &= \frac d{dt}\left(\alpha(\dot\psi(t))\right) \\ &\mbox{This is an arbitrary choice, in which we deviate from Geiges.}
    \end{align*}
    We note that since $\dot\psi(t)$ and $\xi$ build a frame, this completely determines the 1-jet of $h_t$ at time $t$ at the place $\psi(t)$. Geiges' proof involves showing that any Hamiltonian satisfying these equations indeed generates the desired contact isotopy and we refer the reader to his proof for that verification. 
    
    The third condition is actually not a necessary condition to generate the contact isotopy, it could be set to any other number to complete the 1-jet information (Geiges makes a different arbitrary choice $dh_t(\psi(t))[R_\alpha]=0$). We choose the number specifically as we do because we are (as opposed to Geiges in the Theorem 2.6.2) interested in choosing a contact Hamiltonian that is \emph{autonomous}, i.e., independent of time $h_t=h$. For this, it is crucial to recognize that by the Leibniz rule
    $$\frac d{dt} h_t(\psi(t)) = \dot h_t(\psi(t)) + dh_t(\dot\psi(t)).$$
    In order to have first order time independence, we require $\dot h_t(\psi(t))=0$, and by the first equation $\frac d{dt} h_t(\psi(t))=\frac d{dt}\left(\alpha(\dot\psi(t))\right)$ so the Leibniz rule forces the third equation. 
    
    We then extend the 1-jet of $h_t$ along $\psi(t)$ to an actual function in the following way: We choose an auxiliary Riemannian metric on $U$ and map the transverse bundle $\xi|_{\psi(t)}$ to $U$ by the exponential map 
    $$\exp\colon \xi|_{\psi(t)} \to U;\quad (v, t)\mapsto \exp(\psi(t),v).$$
    Restricted to a small enough neighbourhood of the zero section  $V_\xi(\psi(t))$, this map is a diffeomorphism onto its image $V\subseteq U$ (because $\psi$ is transverse to $\xi$). We define $h$ fiberwise affine linearly in the $\xi$-bundle in accordance with the first two equations
    $$\tilde h\colon V_\xi(\psi(t))\to \RR;\quad \tilde h(t, v) = \alpha(\dot\psi(t)) - \dot\psi(t) d\alpha(\dot\psi(t),v)$$
    and then we move that function by the exponential map
    $$h\colon V\to \RR;\quad h(x) = \tilde h (\exp^{-1} x).$$
    Then, $h$ satisfies the first two equations by construction, and third equation is a consequence of being autonomous.

    We thus have constructed our pre-image 
    $$\Rcal_k(j^{k+1}h(\pi_{J^0_\xi(U,1)}\Psi))=\Psi.$$
    
    We could extend $h$ beyond the neighbourhood $V$ by choosing a cutoff function and interpolating to constant 1, but for the purpose of taking the jet this is not necessary. 

    \bparagraph{Submersivity}
    Let $v$ be a vector in $TJ^k_*(U,1)$. We choose a differentiable path $\Psi_s$ such that $\frac{d}{ds}\Psi_s = v$. We are done if we find a differentiable path of contact Hamiltonians $h(s)$ such that $$\Rcal_k(j^{k+1}h_s(\pi_{J^0_\xi(U,1)}\Psi(s)))=\Psi(s).$$ 

    We first note that the associated polynomial representatives $\psi_s(t)$ of $\Psi(s)$ form a differential path of $C^k$ paths since the coefficients of the polynomials are differentiable, since they are coordinates of $\Psi(s)$.

    As a result, the $s$-dependent 1-jet prescription for $h_s$ is differentiable in $s$ and thus also the fiberwise linear extension to actual functions $h_s$ is differentiable in $s$, and thus so is $j^{k+1}h$. The evaluation at the base point $\pi_{J^k_\xi(U,1)}\Psi(s)$ is also differentiable since $\Psi(s)$ is a differentiable path. 
\end{proof}

We thus study the following composition of bundle maps
\[
  \begin{tikzcd}    
    J^{k+1}(\Sigma,\RR_{>0}) \arrow{r}{\Rcal_k} & J^k_\xi(\Sigma,1)\subseteq J_{\pi_Q}^k(U,1) \arrow{r}{\pi^k_Q}& J^k(Q,1).
  \end{tikzcd}
\]
A choice of contact Hamiltonian induces a holonomic section of these bundles which is compatible with the bundle maps.


\subsubsection{Transversality}

We apply the established submersivity to the set of homopodal points. 

\begin{theorem}[Transversality for homopodal points]\label{thm:homopodaltransversality}
    Let $Q$ be a manifold of dimension $d\geq 2$. Let $k\geq 1$. For a $C^{k+2}$-residual set of Hamiltonians $H\in \Ham_s(Q)$, the following is true:
    \begin{itemize}
        \item The set of antipodal points $\Hcal^-_k(\Sigma)\subseteq \Sigma\times\Sigma$ is a submanifold of dimension $(3-k)(n-1)+1$.
        \item The set of non-diagonal isopodal points $\Hcal^+_k\setminus \Delta_\Sigma\subseteq \Sigma\times\Sigma$ is a submanifold of dimension $(3-k)(n-1)+1$. 
    \end{itemize}
\end{theorem}
\begin{proof}
As explained in Lemma~\ref{lem:ReduceToCover}, it is enough to prove the statement for a neighbourhood $\Ucal$ of any  $H\in\Hcal_s$, and any radial neighbourhood $K\subseteq T^*Q$ over a neighbourhood $U\subseteq \Sigma$. Since we look at a pair of jets, for two points $x_1\neq x_2$ in $\Sigma$, we consider two radial neighbourhoods $K_i\subseteq T^*Q$ over neighbourhoods $U_i\subseteq \Sigma$. 

We thus fix $H\in\Ham_s$ and consider only nearby Hamiltonians, which are in the radial neighbourhood $K$ over $U\subseteq \Sigma$ sufficiently modeled by contact Hamitlonians $h\in C^{k+1}(U,\RR_{\geq 0})$. The Homopodal sets $\Hcal_k\cap U_1\times U_2$, which we are continuing to denote by $\Hcal_k$, are then the preimage of the diagonal
$$\Hcal_k=(j^{k+1}h\circ\Rcal_k\circ\pi^k_Q\times j^{k+1}h\circ\Rcal_k\circ\pi^k_Q)^{-1}(\Delta_{J^k(Q,1)})$$
under the map 
$$j^{k+1}h\circ\Rcal_k\circ\pi^k_Q\times j^{k+1}h\circ\Rcal_k\circ\pi^k_Q: U_1\times U_2\subseteq \Sigma\times\Sigma\to J^k(Q,1)\times J^k(Q,1).$$

By design this map factors by first takin a section $$j^{k+1}h\times j^{k+1}h:U_1\times U_2 \to J^{k+1}(U_1,\RR_{>0})\times J^{k+1}(U_2,\RR_{>0})$$
and then continuing with the mapping $$(\Rcal_k\times\Rcal_k)\circ(\pi_Q^k\times\pi_Q^k):J^{k+1}(U_1,\RR_{>0})\times J^{k+1}(U_2,\RR_{>0})\to J^k(Q,1)\times J^k(Q,1).$$
As established above, the composition $\Rcal_k\circ \pi_Q^k$ is a submersion, therefore the same is true for the product map. We conclude that the partial preimage of the diagonal 
$$\DD:=\left((\Rcal_k\times\Rcal_k)\circ(\pi_Q^k\times\pi_Q^k)\right)^{-1}(\Delta_{J^k(Q,1)})$$
is a submanifold (whose dimension we discuss below). Note that $\DD$ is independent of $h$.

By construction, $\Hcal_k=(j^{k+1}h\times j^{k+1}h)^{-1}(\DD)$. But this is the same as intersecting $\DD$ with the graph of the section $j^{k+1}h\times j^{k+1}h$ and projecting back
$$\Hcal_k = \pi_{\Sigma\times\Sigma}((j^{k+1}h\times j^{k+1}h)(U_1\times U_2)\cap \DD).$$
If the intersection $(j^{k+1}h\times j^{k+1}h)(\Sigma\times\Sigma)\cap \DD$ is transverse, then it is a submanifold and then (since $j^{k+1}h\times j^{k+1}h$ is a section) the projection back to the base is a submanifold. Thus we have proved:
\begin{center}
    If for a contact Hamiltonian $h\in C^{k+1}(\Sigma,\RR_{>0})$ the graph of the jet $j^{k+1}h$ is transverse to $\DD$, then $\Hcal_k$ is a sub-manifold of $\Sigma\times\Sigma$. 
\end{center}
All that is left is the proof that for a residual set of $h$ the claimed transversality holds. The following is a modern reformulation of Mather's result on multijet transversality ~\cite{M70}.
\begin{theorem}[Theorem 10.2.1 of ~\cite{G25}]
    Let $E\to M$ be a bundle, let $W\subset (J^{r}E)^s$ be a submanifold in the multijet bundle, and let $K\subset M^{s}\setminus\Delta$ be a closed subset, where $\Delta$ denotes the large diagonal. Consider the set
    $$T_{W,K}:=\{f\in \Gamma^{\infty}E:(j^r)^sf\pitchfork W \mathrm{\:over\:}K\}.$$
    It holds that
    \begin{enumerate}
        \item $T_{W,K}$ is $C^{r+1}$-residual in both the weak and strong Whitney topologies, and
        \item if $W$ is closed and $K$ is compact, then $T_{W,K}$ is open in both the weak and strong Whitney topologies.
    \end{enumerate}
\end{theorem}

\begin{remark}
    The citation~\cite{G25} only concludes $C^{\infty}$-residual, but without changing the proof it also proves $C^{r+1}$-residual.
\end{remark}

For $M=\Sigma$, $E$ the trivial line bundle, $s=2$,  $r=k+1$, $K=\overline{U_1\times U_2}$ and $W=\DD$, this is precisely what we want. It shows transversality for the $k+1$-jet of $h$, and thus holds for a $C^{k+2}$-residual $h$. 

We conclude the proof in the case of transversal intersection by dimension counting: $J^k(Q,1)$ consists of $k$-jets of $\RR^{n-1}$-valued functions and thus is a $1+(k+1)(n-1) = kn+n-k$-dimensional space. Thus, the diagonal $\Delta_{J^k(Q,1)}\subseteq J^k(Q,1)\times J^k(Q,1)$ has codimension $kn+n-k$. Since $\DD$ is the preimage by a submersion, also $\DD\subseteq J^k_*(\Sigma,1)\times J^k_*(\Sigma,1)$ has codimension $kn+n-k$. The dimension of $\Sigma$ is $2n-1$ and thus the graph of $j^{k+1}h\times j^{k+1}h$ has dimension $2(2n-1)$. Finally, since the intersection is transversal, we can compute the dimension of the intersection 
\begin{align*}
    \dim \Hcal_k &=\dim (j^{k+1}h\times j^{k+1}h)(\Sigma\times\Sigma)\cap \DD)\\
    &= \dim (j^{k+1}h\times j^{k+1}h)(\Sigma\times\Sigma) - \codim \DD = 2(2n-1) - (kn+n-k) \\
    &= 3n - 2 - kn + k = (3-k)(n-1)+1,
\end{align*}
which is the claimed dimension. The splitting into $\pm$ flavour works since $k\geq 1$. 
\end{proof}

\begin{corollary}\label{cor:dimensioncount}
    If $n=2$ and $k\geq 5$, or if $n\geq 3$ and $l\geq 4$, then there is a $C^{k+2}$-residual subset of $\Ham_s$ such that the set of $k$-th order homopodes is empty except for the diagonal $\Hcal_k\setminus\Delta_{\Sigma}=\emptyset$.
\end{corollary}
\begin{proof}
    We merely need to estimate the dimension $(3-k)(n-1)+1$. If the dimension is negative, then $\Hcal_k$ is empty except for the diagonal. By Theorem~\ref{thm:homopodaltransversality} transversality can be arranged for a $C^{k+2}-$residual subset of $\Ham_s$.
    
    If $n=2$, then 
    $$(3-k)(n-1)+1=4-k,$$
    which is negative for $k\geq 5$. 

    If $n\geq 3$ and $k\geq 4$, then 
    $$(3-k)(n-1)+1 \leq -1,$$
    which is negative. 
\end{proof}


\end{document}